\documentclass[12pt, reqno]{amsart}

\author[P.~Leonetti]{Paolo Leonetti}
\address{Department of Economics, Universit\`a degli Studi dell'Insubria, via Monte Generoso 71, 21100 Varese, Italy}
\email{leonetti.paolo@gmail.com}

\author[C.~Orhan]{Cihan Orhan}
\address{Department of Mathematics, Faculty of Science, Ankara University, 06100 Tandogan Ankara, Turkey} %
\email{orhan@science.ankara.edu.tr}

\keywords{Generalized limits; 
ideals on $\omega$; 
positive linear functional on $\ell_\infty$; 
positive extensions of limit functional.
}

\subjclass[2010]{
Primary: 47B65, 54D35. Secondary: 40A35, 54A20. 
}

\title{On Generalized Limits and Ultrafilters}

\usepackage[T1]{fontenc}
\usepackage{amsmath}
\usepackage{amssymb}
\usepackage{amsthm}
\usepackage[left=3.5cm, right=3.5cm]{geometry} 
\usepackage{hyperref}
\usepackage{fancyhdr}
\usepackage{enumitem}
\usepackage{comment}
\usepackage{nicefrac}
\usepackage{comment}
\usepackage{bm}
\usepackage{mathrsfs}
\usepackage{graphicx}
\usepackage[utf8]{inputenc}
\usepackage{cancel}
\usepackage{mathtools}

\usepackage{tikz}
\usetikzlibrary{decorations.pathreplacing,matrix,arrows,positioning,automata,shapes,shadows,calc,fadings,decorations,snakes,through,intersections}

\AtBeginDocument{%
   \def\MR#1{}
}

\newtheorem{thm}{Theorem}[section]
\newtheorem{cor}[thm]{Corollary}
\newtheorem{lem}[thm]{Lemma}
\newtheorem{prop}[thm]{Proposition}

\theoremstyle{definition} 
\newtheorem{defi}[thm]{Definition}
\let\olddefi\defi
\renewcommand{\defi}{\olddefi\normalfont}
\newtheorem{question}{Question}
\let\oldquestion\question
\renewcommand{\question}{\oldquestion\normalfont}
\newtheorem{example}[thm]{Example}
\let\oldexample\example
\renewcommand{\example}{\oldexample\normalfont}
\newtheorem{rmk}[thm]{Remark}
\let\oldrmk\rmk
\renewcommand{\rmk}{\oldrmk\normalfont}

\hypersetup{
    pdftitle={},
    pdfauthor={},
    pdfmenubar=false,
    pdffitwindow=true,
    pdfstartview=FitH,
    colorlinks=true,
    linkcolor=blue,
    citecolor=green,
    urlcolor=cyan
}

\providecommand{\MR}[1]{}
\providecommand{\bysame}{\leavevmode\hbox to3em{\hrulefill}\thinspace}
\providecommand{\MR}{\relax\ifhmode\unskip\space\fi MR }

\begin{document}

\maketitle
\thispagestyle{empty}

\begin{abstract} 
Given an ideal $\mathcal{I}$ on $\omega$, 
we denote by $\mathrm{SL}(\mathcal{I})$ the family of positive normalized linear functionals on $\ell_\infty$ which 
assign value $0$ to all 
characteristic sequences of sets in $\mathcal{I}$. 
We show that every element of $\mathrm{SL}(\mathcal{I})$ is a Choquet average of certain ultrafilter limit functionals. Also, we prove that the diameter of $\mathrm{SL}(\mathcal{I})$ is $2$ if and only if $\mathcal{I}$ is not maximal, and that the latter claim can be considerably strengthened if $\mathcal{I}$ is meager. Lastly, we provide several applications: for instance, recovering a result of Freedman in [Bull.~Lond.~Math.~Soc.~\textbf{13} (1981), 224--228], we show that the family of bounded sequences for which all functionals in $\mathrm{SL}(\mathcal{I})$ assign the same value coincides with the closed vector space of bounded $\mathcal{I}$-convergent sequences. 
\end{abstract}


\section{Introduction}\label{sec:intro}

One of the main purpose of operator theory is to give some representations of certain subsets of positive normalized linear functionals on the space of bounded sequences. In the present paper we show, among other things, that every such element is a Choquet average of certain ultrafilter limit functionals.

For the sake of clarity, some notations are in order. Let $\mathcal{I}$ be an ideal on the nonnegative integers $\omega$, that is, a family of subsets of $\omega$ which is stable under finite unions and subsets. Unless otherwise stated, it is also assumed that $\omega\notin \mathcal{I}$, and that the ideal of finite sets $\mathrm{Fin}:=[\omega]^{<\omega}$ is contained in $\mathcal{I}$. Ideals are regarded as subsets of the Cantor space $\{0,1\}^\omega$, hence we can speak about their topological complexity. An ideal $\mathcal{I}$ is a $P$-ideal if it is $\sigma$-directed modulo finite sets, i.e., for every sequence $(A_n: n \in \omega)$ in $\mathcal{I}$ there exists $A \in \mathcal{I}$ such that $A_n\setminus A$ is finite for all $n \in \omega$. For instance, the family of asymptotic density zero sets 
$$
\mathcal{Z}:=\left\{A\subseteq \omega: \lim_{n\to \infty}\frac{|A\cap [0,n]|}{n+1}=0\right\}$$
is a $F_{\sigma\delta}$ $P$-ideal on $\omega$, see e.g. \cite{MR1711328}. 


We consider also the vector space of bounded real sequences $\ell_\infty$ and all its subspaces with the supremum norm and the product pointwise order (so that $\ell_\infty$ is a Banach lattice), and we write $e:=(1,1,\ldots)$. 

The following notion has been introducted by Freedman in \cite[Section 3]{MR614658}:

\begin{defi}\label{defi:SLI}
    A linear functional $f: \ell_\infty\to \mathbf{R}$ is said to be a $\mathrm{S}_{\mathcal{I}}$\emph{-limit} if:
    \begin{enumerate}[label={\rm (\roman{*})}]
    \item \label{item:1SI} $f$ is positive (i.e., $f(x)\ge 0$ for all $x \in \ell_\infty$ with $x\ge 0$); 
    \item \label{item:2SI} $f$ extends the limit functionals (i.e., $f(x)=\lim x$ for all $x \in c$);
    \item \label{item:3SI} $f(\bm{1}_A)=0$ for all $A \in \mathcal{I}$. 
    \end{enumerate}
The family of $\mathrm{S}_{\mathcal{I}}$-limits is denoted by $\mathrm{SL}(\mathcal{I})$. 
\end{defi}


Special instances of Definition \ref{defi:SLI} can be found also in \cite[Definition 1]{MR4109239} and \cite[Definition 1.1]{MR3513160}: in fact, the families $\mathcal{I}$ considered in the latter works are sufficiently \textquotedblleft\,well behaved,\textquotedblright\, cf. e.g. \cite[Proposition 13]{MR3405547} and \cite{MR4041540, MR4756411}. In the case $\mathcal{I}=\mathcal{Z}$, positive linear functionals extending the Ces\`{a}ro mean (which is stronger than property \ref{item:2SI}) have been studied in \cite{MR3278191}. On a different direction, Banach limits are simply elements of $\mathrm{SL}(\mathrm{Fin})$ which are, in addition, translation invariant, see e.g. \cite{MR352932, MR2659770}. 

Some observations are in order. First, the family $\mathrm{SL}(\mathcal{I})$ is nonempty: in fact, if $\mathcal{J}$ is a maximal ideal containing $\mathcal{I}$, then the linear functional $f_{\mathcal{J}}$ defined by 
\begin{equation}\label{eq:fJ}
\forall x \in \ell_\infty, \quad 
f_{\mathcal{J}}(x):=\mathcal{J}\text{-}\lim x
\end{equation}
is a $\mathrm{S}_{\mathcal{I}}$-limit (it is well known that $f_{\mathcal{J}}$ is well defined). 

Second, properties \ref{item:1SI}--\ref{item:3SI} in Definition \ref{defi:SLI} are independent of each other. For, if $\mathcal{I}\neq \mathrm{Fin}$, pick an infinite $A \in \mathcal{I}$ (hence, $A$ is not cofinite) and let $\mathcal{J}$ be a maximal ideal containing $\mathrm{Fin} \cup \{\omega\setminus A\}$. Then the linear operator $f_{\mathcal{J}}$ defined in \eqref{eq:fJ} satisfies \ref{item:1SI} and \ref{item:2SI}, but property \ref{item:3SI} fails. Also, trivially, $f=0$ satisfies \ref{item:1SI} and \ref{item:3SI}, but property \ref{item:2SI} fails. Lastly, to prove the independence of property \ref{item:1SI}, suppose that $\mathcal{I}$ is not maximal, hence it is possible to pick two distinct maximal ideals $\mathcal{J}_1$, $\mathcal{J}_2$ containing $\mathcal{I}$. Fix $A\in \mathcal{J}_1\setminus \mathcal{J}_2$. Then $f:=f_{\mathcal{J}_1}-f_{\mathcal{J}_2}$ satisfies \ref{item:2SI} and \ref{item:3SI}, while property \ref{item:1SI} fails since $f(\bm{1}_A)=-1$.




Third, it is easy to see that, if $f$ is a $\mathrm{S}_{\mathcal{I}}$-limit, then
\begin{enumerate}[label={\rm (\roman{*}$^\prime$)}]
    \item \label{item:1SIprime} $f$ is continuous and $\|f\|=f(e)=1$;
\end{enumerate}
cf. e.g. \cite[Theorem 4.3]{MR2262133}. In fact, also the converse holds, namely, if $f$ is a linear functional on $\ell_\infty$ satisfying \ref{item:1SIprime}, \ref{item:2SI}, and \ref{item:3SI}, then also \ref{item:1SI} holds. For, suppose that $f$ is not positive, hence there exists $x\ge 0$ such that $\alpha:=f(x)<0$. Pick a sufficiently small $t>0$ such that $\|e-t x\|\in (0,1]$. Then 
$$
1=\|f\| \ge f(e-t x)/\|e-tx\|\ge f(e-tx)=1-t\alpha>1,
$$
which provides the required contradiction; cf. also Remark \ref{rmk:equivalentdefi} below for additional equivalent definitions of $S_{\mathcal{I}}$-limits. 
We remark that a linear functional $f: \ell_\infty\to \mathbf{R}$ satisfying \ref{item:1SIprime} and \ref{item:2SI} has been termed \textquotedblleft extended limit\textquotedblright\, by Bennett and Kalton in \cite[Section 2]{MR352932}.

With the above premises, following Freedman \cite[Section 1]{MR614658}, we consider the vector space $\mathscr{V}(\mathcal{I})$ of sequences $x \in \ell_\infty$ for which all $S_{\mathcal{I}}$-limits assign the same value: 
\begin{defi}\label{defi:VI}
Given an ideal $\mathcal{I}$ on $\omega$, define
$$
\mathscr{V}(\mathcal{I}):=\left\{x \in \ell_\infty:  f(x)=g(x) \text{ for all }f,g \in \mathrm{SL}(\mathcal{I})\right\}.
$$
\end{defi}

Given an ideal $\mathcal{I}$ on $\omega$, a (not necessarily bounded) real sequence $x \in \mathbf{R}^\omega$ is said to be $\mathcal{I}$\emph{-convergent to} $\eta \in \mathbf{R}$, shortened as $\mathcal{I}\text{-}\lim x=\eta$, if $\{n \in \omega: |x_n-\eta|\ge \varepsilon\} \in \mathcal{I}$ for all $\varepsilon>0$. The vector space of $\mathcal{I}$-convergent sequences is denoted by $c(\mathcal{I})$. Let also $c_{00}(\mathcal{I})$ be the set of sequences which are supported on $\mathcal{I}$. 

Similarly, a real sequence $x \in \mathbf{R}^\omega$ is said to be $\mathcal{I}^\star$\emph{-convergent to} $\eta \in \mathbf{R}$, shortened as $\mathcal{I}^\star\text{-}\lim x=\eta$, if there exists $A \in \mathcal{I}$ such that the subsequence $(x_n: n \in \omega\setminus A)$ is convergent (in the ordinary sense) to $\eta$. The vector space of $\mathcal{I}$-convergent sequences is denoted by $c(\mathcal{I}^\star)$. It is folklore that $\mathcal{I}^\star$ convergence is stronger than $\mathcal{I}$-convergence, and that they coincide if and only if $\mathcal{I}$ is a $P$-ideal. 

Freedman proved in \cite{MR614658} that $\mathscr{V}(\mathcal{I})$ coincides with the closure of the vector space of bounded $\mathcal{I}^\star$-convergent sequences, namely, $\mathscr{V}(\mathcal{I})=\overline{c(\mathcal{I}^\star) \cap \ell_\infty}$. Here, as an application of a representation of $S_{\mathcal{I}}$-limits as Choquet averages of the functionals $f_{\mathcal{J}}$ defined in \eqref{eq:fJ} (see Theorem \ref{thm:firstrepresenation} below for details), we recover Freedman's result through a simpler and direct proof. 
\begin{thm}\label{thm:freedman}
    Let $\mathcal{I}$ be an ideal on $\omega$. Then $\mathscr{V}(\mathcal{I})=\overline{c(\mathcal{I}^\star) \cap \ell_\infty}=c(\mathcal{I}) \cap \ell_\infty$. 
    
    If, in addition, $\mathcal{I}$ is a $P$-ideal, then also $\mathscr{V}(\mathcal{I})=c+(c_{00}(\mathcal{I})\cap \ell_\infty)$. 
\end{thm}

In our main results, we will provide characterizations of $\mathrm{SL}(\mathcal{I})$, some structural properties, and related results. The proof of Theorem \ref{thm:freedman} is given at the end of Section \ref{sec:proofs}.

\section{Main results}\label{sec:mainresults}

Consider the Stone--\u{C}ech compactification $\beta\omega$ of the nonnegative integers $\omega$, and recall that it is homeomorphic to the space of ultrafilters $\mathcal{F}$ on $\omega$, which we still denote by $\beta\omega$ and is topologised by the base of clopen subsets $\{\{\mathcal{F} \in \beta\omega: A \in \mathcal{F}\}: A\subseteq \omega\}$. For each ultrafilter $\mathcal{F}$, we write 
$
\mathcal{J}_{\mathcal{F}}:=\{A\subseteq \omega: \omega\setminus A \in \mathcal{F}\}
$ 
for its associated maximal ideal. 
To ease the notation, we will use $\mathcal{J}_{\mathcal{F}}\text{-}\lim x$ or ${\mathcal{F}}\text{-}\lim x$ interchangeably, and, similarly, $f_{\mathcal{F}}$ in place of $f_{\mathcal{J}_{\mathcal{F}}}$. 
Also, by $\mathrm{Ult}(\mathcal{I})$ we denote the compact subspace of free ultrafilters  which contain the dual filter of a given ideal $\mathcal{I}$, that is, equivalently, 
$$
\mathrm{Ult}(\mathcal{I}):=\{\mathcal{F} \in \beta\omega: \mathcal{I}\subseteq \mathcal{J}_{\mathcal{F}}\}.
$$
The space $\mathrm{Ult}(\mathcal{I})$ is endowed with its relative topology. 
Note that the subspace $\mathrm{Ult}(\mathcal{I})$, sometimes called \textquotedblleft support set,\textquotedblright
\hspace{.5mm}has been introduced in Henriksen \cite{MR108720} and further studied in \cite{MR271752, MR1095221, 
MR1372186, MR279609} in the context of ideals generated by nonnegative regular summability matrices, cf. also \cite{MR3405547, MR4693672, MR3883309}.

In our first representation of $\mathrm{SL}(\mathcal{I})$, we need to recall the notion of Choquet integral. For, given a measurable space $(S,\Sigma)$ and a normalized capacity $\nu: \Sigma\to \mathbf{R}$, we define the \emph{Choquet integral} of a bounded $\Sigma$-measurable function $x: S\to \mathbf{R}$ with respect to $\nu$ as the quantity 
$$
\int x\,\mathrm{d}\nu:=\int_0^\infty \nu(x\ge t)\,\mathrm{d}t+\int_{-\infty}^0 [\nu(x\ge t)-\nu(S)]\,\mathrm{d}t,
$$
where the integrals on the right hand side are meant to be improper Riemann
integrals. This naturally generalizes the standard notion of integral since the two coincide when $\nu$ is finitely additive, see \cite{MR4693672} and references therein. 

In addition, given a topological space $X$, denote by $\mathscr{B}(X)$ its Borel $\sigma$-algebra (recall that if a subspace $Y\subseteq X$ is endowed with its relative topology then $\mathscr{B}(Y)=\{A\cap Y: A \in \mathscr{B}(X)\}$).
\begin{thm}\label{thm:firstrepresenation}
Let $\mathcal{I}$ be an ideal on $\omega$. Then, for each $f \in \mathrm{SL}(\mathcal{I})$, there exists a normalized capacity $\rho: \mathscr{B}(\mathrm{Ult}(\mathcal{I}))\to \mathbf{R}$ such that 
\begin{equation}\label{eq:firstrepresetnation}
\forall x \in \ell_\infty, \quad 
f(x)=\int_{\mathrm{Ult}(\mathcal{I})} {\mathcal{F}}\text{-}\lim x \, \mathrm{d}\rho(\mathcal{F}). 
\end{equation}
\end{thm}

In our second representation of $\mathrm{SL}(\mathcal{I})$, we endow the norm dual $\ell_{\infty }^{\,\prime}$ with the weak%
$^{\star}$ topology. Hence, for each $Y\subseteq \ell_{\infty }^{\,\prime}$, $\overline{\mathrm{co}}(Y)$ stands for the the weak%
$^{\star}$ closed convex hull of $Y$. 
Recall also that
$f_{\mathcal{F}}(x):=\mathcal{F}\text{-}\lim x$ for each $x \in \ell_\infty$ and $\mathcal{F} \in \beta\omega$.
\begin{thm}\label{thm:secondrepresenation}
Let $\mathcal{I}$ be an ideal on $\omega$. Then
    $$
    \mathrm{SL}(\mathcal{I})=\overline{\mathrm{co}}(\{f_{{\mathcal{F}}}: \mathcal{F} \in \mathrm{Ult}(\mathcal{I})\}).
    $$
\end{thm}

 Since each $S_{\mathcal{I}}$-limit has norm one, it follows that the diameter of $\mathrm{SL}(\mathcal{I})$, that is, 
$
\mathrm{diam}(\mathrm{SL}(\mathcal{I})):=\sup\{\|f-g\|: f,g \in \mathrm{SL}(\mathcal{I})\}
$ 
is at most $2$. In our next result, we show that this upper bound is optimal in the nonmaximal case. 
\begin{thm}\label{thm:weakdiameter}
Let $\mathcal{I}$ be an ideal on $\omega$. Then $\mathrm{diam}(\mathrm{SL}(\mathcal{I}))=2$ if and only if $\mathcal{I}$ is not maximal. 
\end{thm}

If, in addition, the ideal $\mathcal{I}$ is meager (i.e., it can be regarded as a meager subset of $\{0,1\}^\omega$), then the above claim can be strenghtened.
\begin{thm}\label{thm:strongdiameter}
    Let $\mathcal{I}$ be a meager ideal on $\omega$ and fix $f \in \mathrm{SL}(\mathcal{I})$. Then there exists $g \in \mathrm{SL}(\mathcal{I})$ such that $\|f-g\|=2$. 
\end{thm}

In the next result, we characterize the set of differences $\mathrm{SL}(\mathcal{I})-\mathrm{SL}(\mathcal{I})$:
\begin{thm}\label{thm:differenceSIlimits}
    Let $\mathcal{I}$ be an ideal on $\omega$ and fix $f \in \ell_\infty^{\,\prime}$. Then there exist $g,h \in \mathrm{SL}(\mathcal{I})$ such that $f=g-h$ if and only if the following conditions hold\textup{:}
    \begin{enumerate}[label={\rm (\alph{*})}]
    \item \label{item:aSLSL} $f(e)=0$\textup{;}
    \item \label{item:bSLSL} $\|f\|\le 2$\textup{;}
    \item \label{item:cSLSL} $f(\bm{1}_A)=0$ for all $A \in \mathcal{I}$\textup{.}
    \end{enumerate}
In addition, such decomposition is unique if and only if $\|f\|=2$ or $\mathcal{I}$ is maximal. 
\end{thm}

Proofs are given in Section \ref{sec:proofs}. Further results are given in Section \ref{sec:furtherresults}. 



\section{Preliminaries}\label{sec:preliminaries}


Given an ideal $\mathcal{I}$ on $\omega$ and a bounded real sequence $x \in \ell_\infty$, a point $\eta \in \mathbf{R}$ is said to be a $\mathcal{I}$\emph{-cluster point} of $x$  if $\{n \in \omega: |x_n-\eta|\le \varepsilon\} \notin \mathcal{I}$ for all $\varepsilon>0$ (note that they have been termed \emph{statistical cluster point} if $\mathcal{I}=\mathcal{Z}$, see e.g. \cite{MR1181163, MR1416085}). 
For each $x \in \ell_\infty$, we denote by $\Gamma_x(\mathcal{I})$ the set of its of $\mathcal{I}$-cluster points; in addition, $\mathcal{I}\text{-}\liminf x:=\min \Gamma_x(\mathcal{I})$ and $\mathcal{I}\text{-}\limsup x:=\max \Gamma_x(\mathcal{I})$, which are well defined since $\Gamma_x(\mathcal{I})$ is nonempty compact; 
see e.g. \cite{MR3920799} and references therein for basic properties and characterization of $\mathcal{I}$-cluster points. Interestingly, the map $x\mapsto \mathcal{I}\text{-}\limsup x$ provides a complete pseudonorm on $\ell_\infty$, see \cite{MR2735533, MR4610955}. 

Recall also that each sequence $x\in \ell_\infty$ can be extended uniquely to a continuous function $\hat{x}: \beta\omega \to \mathbf{R}$ defined by 
$$
\forall \mathcal{F} \in \beta\omega, \quad 
\hat{x}(\mathcal{F}):=\mathcal{F}\text{-}\lim x.
$$
In particular, if $\mathcal{F}$ is the principal ultrafilter associated with an integer $n \in \omega$, then $\hat{x}(\mathcal{F})=x_n$. Hereafter, for each $A\subseteq \omega$, we write 
$$
\tilde{A}:=\{\mathcal{F} \in \beta\omega: A \in \mathcal{F} \text{ and }\mathcal{F}\text{ free ultrafilter }\}.
$$

\begin{lem}\label{lem:bndensityzero}
Let $\mathcal{I}$ be an ideal on $\omega$. Then, for each $A\subseteq \omega$, we have 
$A \in \mathcal{I}$ if and only if $\tilde{A} \cap \mathrm{Ult}(\mathcal{I})=\emptyset$. 
\end{lem}
\begin{proof}
Fix $A\subseteq \omega$ and define the dual filter $\mathcal{G}:=\{B \subseteq \omega: \omega\setminus B \in \mathcal{I}\}$. First, suppose that $A \in \mathcal{I}$ and let $\mathcal{F}$ be a free ultrafilter containing $A$, i.e., $\mathcal{F} \in \tilde{A}$. Then $\omega\setminus A \notin \mathcal{F}$ while, on the other hand, it belongs to $\mathcal{G}$. Hence $\tilde{A} \cap \mathrm{Ult}(\mathcal{I})=\emptyset$. 
%

Conversely, suppose that $A \notin \mathcal{I}$. Considering that every finite subfamily of $\{A\} \cup \mathcal{G}$ has intersection not in $\mathcal{I}$ (hence, infinite intersection), there exists a free ultrafilter $\mathcal{F}$ containing $A$ and $\mathcal{G}$, hence $\tilde{A} \cap \mathrm{Ult}(\mathcal{I})\neq \emptyset$.
\end{proof}

\begin{lem}\label{lem:bnsubseqconvergence}
Let $\mathcal{I}$ be an ideal on $\omega$ and fix an infinite set $A\subseteq \omega$ with increasing enumeration $(a_n: n \in \omega)$. Then, for each $x \in \ell_\infty$ and $\eta \in \mathbf{R}$, we have $\lim_n x_{a_n}=\eta$ if and only if $\hat{x}[\,\tilde{A}\,]=\{\eta\}$.
\end{lem}
\begin{proof}
Suppose that $\lim_n x_{a_n}=\eta$ and fix a free ultrafilter $\mathcal{F}$ containing $A$, i.e., $\mathcal{F} \in \tilde{A}$. Then for each neighborhood $U$ of $\eta$ there exists $F \in \mathrm{Fin}$ such that $\{n\in \omega: x_n \in U\}\supseteq A\setminus F$, hence it belongs to $\mathcal{F}$. Therefore $\hat{x}(\mathcal{F})=\mathcal{F}\text{-}\lim x=\eta$. 

Conversely, suppose that $\hat{x}[\,\tilde{A}\,]=\{\eta\}$ and let $U$ be a neighborhood of $\eta$. Then $\{n\in \omega: x_{n} \in U\} \in \mathcal{F}$ for each $F \in \tilde{A}$. Hence $\{n\in \omega: x_{n} \in U\}$ belongs to $\bigcap \tilde{A}$, which is known to be the filter $\{B \subseteq \omega: A\setminus B\in \mathrm{Fin}\}$. We conclude that the subsequence $(x_{a_n}:n \in \omega)$ converges to $\eta$. 
\end{proof}

As a consequence, we provide a characterization of $\mathcal{I}$-cluster points, which generalizes \cite[Theorem 2(3)]{MR1372186}. 
\begin{thm}\label{thm:Iclusterpoints}
Let $\mathcal{I}$ be an ideal on $\omega$. Then $\Gamma_x(\mathcal{I})=\hat{x}[ \mathrm{Ult}(\mathcal{I})]$ for each $x \in \ell_\infty$. 
\end{thm}
\begin{proof}
Fix $x \in \ell_\infty$. First, suppose that there exists a free ultrafilter $\mathcal{F} \in \mathrm{Ult}(\mathcal{I})$ such that $\hat{x}(\mathcal{F})=\eta$, that is, $\mathcal{F}\text{-}\lim x=\eta$. Hence, for each neighborhood $U$ of $\eta$, it holds $\{n\in\omega: x_n \in U\} \in \mathcal{F}$. Considering that $\mathcal{F} \cap \mathcal{I}=\emptyset$, we obtain that $\eta$ is an $\mathcal{I}$-cluster point of $x$. Hence $\hat{x}[ \mathrm{Ult}(\mathcal{I})]\subseteq \Gamma_x(\mathcal{I})$. 

Conversely, suppose that $\eta \in \Gamma_x(\mathcal{I})$ and let $(U_k)$ be a decreasing local base of neighborhoods at $\eta$. Then $A_k:=\{n\in \omega: x_n \in U_k\} \notin \mathcal{I}$ for each $k\in \omega$. Thanks to Lemma \ref{lem:bndensityzero}, for each $k\in \omega$ there exists a free ultrafilter $\mathcal{F}_k$ containing $A_k$ and the dual filter of $\mathcal{I}$. In particular, $\hat{x}(\mathcal{F}_k)=\eta$ for each $k\in \omega$. Since $(\mathcal{F}_k: k \in \omega)$ is a sequence in the compact space $\mathrm{Ult}(\mathcal{I})$, there is a subsequence of ultrafilters converging to some $\mathcal{F} \in \mathrm{Ult}(\mathcal{I})$. By the continuity of $\hat{x}$, we conclude that $\hat{x}(\mathcal{F})=\eta$. Therefore $\Gamma_x(\mathcal{I})\subseteq \hat{x}[ \mathrm{Ult}(\mathcal{I})]$. 
\end{proof}

As a consequence, we recover an intermediate result of Freedman \cite[p. 226]{MR614658}: 
\begin{cor}
Let $\mathcal{I}$ be an ideal on $\omega$ and fix $A\subseteq \omega$ with $A\notin \mathcal{I}$. Then there exists $f \in \mathrm{SL}(\mathcal{I})$ such that $f(\bm{1}_A)=1$. 
\end{cor}
\begin{proof}
    Since $A\notin \mathcal{I}$ then $1$ is an $\mathcal{I}$-cluster of the sequence $\bm{1}_A$. We obtain by Theorem \ref{thm:Iclusterpoints} that there exists $\mathcal{F} \in \mathrm{Ult}(\mathcal{I})$ such that $f_{\mathcal{F}}(\bm{1}_A)=\mathcal{F}\text{-}\lim \bm{1}_A=1$. The claim follows by the fact that $f_{\mathcal{F}} \in \mathrm{SL}(\mathcal{I})$, as anticipated in Section \ref{sec:intro}.  
\end{proof}

As another consequence, we generalize a claim contained in \cite[Theorem 3.2(i)]{MR3513160} and \cite[Theorem 2.4(i)]{CihanMustafa25}:
\begin{cor}\label{cor:liminflimsup}
Let $\mathcal{I}$ be an ideal on $\omega$. Then 
$$
\forall x \in \ell_\infty, \quad 
\{f(x): f \in \mathrm{SL}(\mathcal{I})\}=\mathrm{co}(\Gamma_x(\mathcal{I})).
$$
In particular, $\mathcal{I}\text{-}\liminf x\le f(x)\le \mathcal{I}\text{-}\limsup x$ for all $f \in \mathrm{SL}(\mathcal{I})$ and $x \in \ell_\infty$. 
\end{cor}
\begin{proof}
    It follows by Theorem \ref{thm:secondrepresenation} and Theorem \ref{thm:Iclusterpoints}. 
\end{proof}


We conclude with a computation of the norm of $f \in \ell_\infty^{\,\prime}$ such that $f(e)=0$. 
\begin{lem}\label{lem:abramoich}
    Fix $f \in \ell_\infty^{\,\prime}$ with $f(e)=0$. Then 
    $$
    \|f\|=2\sup\{f(x): 0\le x\le e\}=2\sup\{f(\bm{1}_A): A\subseteq \omega\}
    .$$ 
\end{lem}
\begin{proof}
Set $s:=\sup\{f(x): 0\le x\le e\}$. On the one hand, we have 
\begin{displaymath}
    \begin{split}
        \|f\|&=\sup\{f(x): -e\le x\le e\}\\
        &=\sup\{f(a-b): 0\le a,b\le e, a\wedge b=0\}\\
        &\le \sup\{f(a)+f(-b): 0\le a,b\le e\}\\
        &= \sup\{f(a)+f(e-b): 0\le a,b\le e\}\\
        &= \sup\{f(a)+f(b): 0\le a,b\le e\}\le 2s.
    \end{split}
\end{displaymath}
On the other hand, we have that
\begin{displaymath}
    \begin{split}
         \|f\|&\ge \sup\{f(\bm{1}_A-\bm{1}_{\omega\setminus A}): A\subseteq \omega\}\\
         &=\sup \{f(\bm{1}_A)+f(e-\bm{1}_{\omega\setminus A}): A\subseteq \omega\}\\
         &=2\sup \{f(\bm{1}_A): A\subseteq \omega\}. 
    \end{split}
\end{displaymath}
To conclude, it is enough to observe that the equality $\sup \{f(\bm{1}_A): A\subseteq \omega\}=s$ follows from a result of Abramovich, see \cite[Theorem 1.50]{MR2262133} and \cite[p. 541]{MR2378491}.
\end{proof}


\section{Proofs of Main results}\label{sec:proofs}

In this section, we provide the proofs of our main results. 
\begin{lem}\label{lem:Iinvariant}
    Let $\mathcal{I}$ be an ideal on $\omega$ and fix $f \in \mathrm{SL}(\mathcal{I})$ and $A,B \subseteq \omega$ with $A\bigtriangleup B \in \mathcal{I}$. Then $f(\bm{1}_A)=f(\bm{1}_B)$. 
\end{lem}
\begin{proof}
    This follows by $0\le |f(\bm{1}_A)-f(\bm{1}_B)|\le f(\bm{1}_{A\bigtriangleup B})=0$.
\end{proof}

Hereafter, we denote by $ba$ the space of signed finitely additive measures $\mu :\mathcal{P}(\omega)\to \mathbf{R}$ with finite total variation, that is, such that 
$$
\sup \left\{ \sum\nolimits_{i=0}^{n}|\mu (A_{i})|:\{A_{0},A_1,\ldots ,A_{n}\}%
\text{ is a partition of }\omega\right\} <\infty,
$$
see e.g. \cite[Section 10.10]{MR2378491}. 
Recall that the norm dual $\ell _{\infty}^{\,\prime}$ can be identified with $ba$ via the lattice isomorphism $T:\ell _{\infty}^{\,\prime}\to ba$ defined by%
\begin{equation}\label{eq:latticeisomorphism}
\forall f \in \ell _{\infty}^{\,\prime},\forall A\subseteq \omega,\quad
T(f)(A):=f(\bm{1}_{A}),
\end{equation}
see e.g. \cite[Theorem 14.4]{MR2378491}. 
We endow both $ba$ and the norm dual $\ell _{\infty }^{\,\prime}$ with the weak%
$^{\star }$ topology. Note that $T$ is continuous and its inverse is given by%
\begin{equation}\label{eq:Tinverserecall}
\forall \mu \in ba,\forall x\in \ell _{\infty },\quad 
T^{-1}(\mu)(x)=\int_{\mathbf{\omega}}x\,\mathrm{d}\mu . 
\end{equation}


\begin{rmk}\label{rmk:equivalentdefi}
   By the above premises, it is easy to check that a linear continuous functional $f \in \ell_\infty^{\,\prime}$ belongs to $\mathrm{SL}(\mathcal{I})$ if and only if the associated finitely additive measure $\mu:=T(f)$ is nonnegative, normalized, and assigns $0$ to all sets in $\mathcal{I}$ (or, equivalently, it is $\mathcal{I}$-invariant). 

   This implies that 
   item \ref{item:2SI} in Definition \ref{defi:SLI} can be replaced by $f(e)=1$ (in fact, the continuity of $f$ follows by item \ref{item:1SI}). In addition, item \ref{item:3SI} can be replaced by the equivalent (formally stronger) $f(x)=0$ for all 
   $x \in c_{00}(\mathcal{I})\cap \ell_\infty$. 
%
   Lastly, it is easy to see that $\mathrm{SL}(\mathcal{I})$ coincides also with the set of positive linear functionals $f$ such that $f(x)=\mathcal{I}\text{-}\lim x$ for all $x \in c(\mathcal{I})\cap \ell_\infty$, cf. also Corollary \ref{cor:liminflimsup} below. 
\end{rmk}

\begin{proof}[Proof of Theorem \ref{thm:firstrepresenation}]
Fix $f \in \mathrm{SL}(\mathcal{I})$ and recall that $f$ is continuous since it satisfies \ref{item:1SIprime}. Using the lattice isomorphism $T$ defined in \eqref{eq:latticeisomorphism}, let $\mu:=T(f)$ be the nonnegative finitely additive normalized measure $\mathcal{P}(\omega)\to \mathbf{R}$ associated with $f$. 
Note that $\mu$ is $\mathcal{I}$-invariant, thanks to Lemma \ref{lem:Iinvariant}. 
It follows by \eqref{eq:Tinverserecall} that $f(x)=\int_{\omega}x\,\mathrm{d}\mu$ for all $x \in \ell_\infty$. 
Also, for each ultrafilter $\mathcal{F}$ on $\omega$, let $\mu_{\mathcal{F}}$ be the $\{0,1\}$-valued finitely additive probability measure defined by $\mu_{\mathcal{F}}(A)=1$ if and only if $A \in \mathcal{F}$. 
Taking again into account \eqref{eq:Tinverserecall}, we obtain by \cite[Theorem 1.1]{MR4693672} that there exists a normalized capacity $\rho: \mathscr{B}(\mathrm{Ult}(\mathcal{I}))\to \mathbf{R}$ such that 
$$
f(x)=\int_{\mathbf{\omega}}x\,\mathrm{d}\mu
=\int_{\mathrm{Ult}(\mathcal{I})} \left(\int_{\omega}x\,\mathrm{d}\mu_{\mathcal{F}}\right) \, \mathrm{d}\rho(\mathcal{F})
=\int_{\mathrm{Ult}(\mathcal{I})} {\mathcal{F}}\text{-}\lim x \, \mathrm{d}\rho(\mathcal{F})
$$
for each $x \in \ell_\infty$. This completes the proof. 
\end{proof}

\medskip

\begin{proof}[Proof of Theorem \ref{thm:secondrepresenation}]
    It is clear that $\mathrm{SL}(\mathcal{I})$ is a weak$^{\star}$ closed convex subset of $\ell_{\infty }^{\,\prime}$. In addition, since it is included in the closed unit ball of $\ell_\infty^{\,\prime}$, it is weak$^\star$ compact by Alaoglu's theorem, see e.g. \cite[Theorem 5.105]{MR2378491}. By the lattice isomorphism $T$ defined in \eqref{eq:latticeisomorphism}, we have that 
    $$
    \mathrm{ext}(\mathrm{SL}(\mathcal{I}))
    =T^{-1}\left[\,\mathrm{ext}\left(T[\mathrm{SL}(\mathcal{I})]\right)\right]
    =T^{-1}\left[\{\mu_{\mathcal{F}}: \mathcal{F} \in \mathrm{Ult}(\mathcal{I})\}\right],
    $$
    see \cite[Claim 4]{MR4693672} and cf. \cite[p. 544]{MR2378491} (here, as usual, $\mathrm{ext}(S)$ stands for the set of extreme points of a nonempty $S\subseteq \ell_\infty^{\,\prime}$).  
    Hence by the inverse $T^{-1}$ in \eqref{eq:Tinverserecall} we obtain $\mathrm{ext}(\mathrm{SL}(\mathcal{I}))=\{f_{{\mathcal{F}}}: \mathcal{F} \in \mathrm{Ult}(\mathcal{I})\}$. The conclusion follows by applying Krein--Milman's theorem, see e.g. \cite[Theorem 7.68]{MR2378491}.
\end{proof}

\medskip

\begin{proof}[Proof of Theorem \ref{thm:weakdiameter}]
    If $\mathcal{I}$ is maximal then it follows by Theorem \ref{thm:secondrepresenation} that $\mathrm{SL}(\mathcal{I})=\{f_{\mathcal{I}}\}$, hence $\mathrm{diam}(\mathrm{SL}(\mathcal{I}))=0$. Conversely, if $\mathcal{I}$ is not maximal, there exists a set $A\subseteq \omega$ such that $A\notin \mathcal{I}$ and $\omega\setminus A\notin \mathcal{I}$. Define $x:=\bm{1}_A-\bm{1}_{\omega\setminus A}$, so that $\Gamma_x(\mathcal{I})=\{-1,1\}$. It follows by Theorem \ref{thm:secondrepresenation} and Theorem \ref{thm:Iclusterpoints} that there exist $f,g \in \mathrm{SL}(\mathcal{I})$ such that $f(x)=1$ and $g(x)=-1$. Therefore $\mathrm{diam}(\mathrm{SL}(\mathcal{I}))\ge \|f-g\|\ge |f(x)-g(x)|=2$. 
\end{proof}

\medskip

\begin{proof}
[Proof of Theorem \ref{thm:strongdiameter}]
    Since $f \in \mathrm{SL}(\mathcal{I})$, it follows by definition that $\mathcal{I}\subseteq \{A\subseteq \omega: f(\bm{1}_A)=0\}$. On the other hand, thanks to \cite[Corollary 2.10]{MR4358610}, the latter inclusion has to be strict, hence there exists $A\subseteq \omega$ such that $A\notin \mathcal{I}$ and $f(\bm{1}_A)=0$. Since $f(\bm{1}_\omega)=1$, we get by Lemma \ref{lem:Iinvariant} that also $\omega\setminus A\notin \mathcal{I}$. At this point, define the sequence 
    $
    x:=\bm{1}_A-\bm{1}_{\omega\setminus A}.
    $ 
    Then $\Gamma_{x}(\mathcal{I})=\{1,-1\}$ and $f(x)=f(\bm{1}_A)-f(\bm{1}_{\omega\setminus A})=-1$. It also follows by Theorem \ref{thm:secondrepresenation} that there exists $g \in \mathrm{SL}(\mathcal{I})$ such that $g(x)=1$. Therefore $\|f-g\|\ge |f(x)-g(x)|=2$. 
\end{proof}

\begin{rmk}\label{rmk:kadets}
As it follows by \cite[Proposition 2.11]{MR4358610} and the proof above, there exists an ideal $\mathcal{I}$ on $\omega$ which is \emph{not} meager and, on the other hand, satisfies the claim of Theorem \ref{thm:strongdiameter}. 
\end{rmk}

\medskip

Before we proceed with the proof of Theorem \ref{thm:differenceSIlimits}, 
recall that the topological dual $\ell_\infty^{\,\prime}$ is a Dedekind complete Riesz spaces, and its partial order satisfies 
\begin{equation}\label{eq:fveeg}
(f\vee g)(x)=\sup\{f(u)+g(v): x=u+v, u,v\ge 0\}  
\end{equation}
for all $f,g \in \ell_\infty^{\,\prime}$ and all $x \in \ell_\infty^+$, 
see e.g. \cite[Theorem 1.18 and Theorem 3.49]{MR2262133}. 

\begin{proof}
[Proof of Theorem \ref{thm:differenceSIlimits}]
    It is clear that each element of $\mathrm{SL}(\mathcal{I})-\mathrm{SL}(\mathcal{I})$ satisfies properties \ref{item:aSLSL}--\ref{item:cSLSL}. Conversely, suppose that $f$ satisfies properties \ref{item:aSLSL}--\ref{item:cSLSL}. 
    Fix an arbitrary $f_0 \in \mathrm{SL}(\mathcal{I})$, and define
    \begin{equation}\label{eq:definitiong}
    g:=(f \vee 0)+\left(1-\frac{\|f\|}{2}\right)f_0,
    \end{equation}
    where the supremum is computed as in \eqref{eq:fveeg}. Of course, $g$ is a continuous linear functional on $\ell_\infty$. We observe that $g$ is positive, indeed for every $x\ge 0$ we have $(f\vee 0)(x)\ge 0$, $f_0(x) \ge 0$, and $1-\|f\|/2\ge 0$, hence $g(x)\ge 0$.  
    In addition, for every $A \in \mathcal{I}$ we have  
    $f_0(\bm{1}_A)=0$ and $(f\vee 0)(\bm{1}_A)=\sup\{f(x): 0\le x\le \bm{1}_A\}=\sup\{f(\bm{1}_B): B\subseteq A\}=0$, cf. again \cite[Theorem 1.50]{MR2262133} and \cite[p. 541]{MR2378491}. 
    Lastly, we observe by Lemma \ref{lem:abramoich} that
    \begin{displaymath}
        \begin{split}
            g(e)&=(f\vee 0)(e)+\left(1-\frac{\|f\|}{2}\right)f_0(e)\\
            &=\sup\{f(x): 0\le x\le e\}+\left(1-\frac{\|f\|}{2}\right)=1. 
        \end{split}
    \end{displaymath}
    This implies, thanks to 
    Remark \ref{rmk:equivalentdefi}, that $g \in \mathrm{SL}(\mathcal{I})$. 
    
    At this point, it is enough to show that 
    $$
    h:=g-f=((-f)\vee 0)+\left(1-\frac{\|f\|}{2}\right)f_0
    $$
    belongs to $\mathrm{SL}(\mathcal{I})$ as well. The proof goes analogously, by noting that $\sup\{-f(x): 0\le x\le e\}$ can be rewritten equivalently as $\sup\{f(e-x): 0\le x\le e\}$ or $\sup\{f(x): 0\le x\le e\}$ (we omit further details). 

    \medskip
    
    For the second part, suppose first that $\mathcal{I}$ is maximal. Note that there exists exactly one $\mathrm{S}_{\mathcal{I}}$-limit, which is $f_0(x):=\mathcal{I}\text{-}\lim x$ for all $x \in \ell_\infty$, cf. Theorem \ref{thm:secondrepresenation}. Then a linear functional $f \in \ell_\infty^{\,\prime}$ which satisfies properties \ref{item:aSLSL}--\ref{item:cSLSL} is necessarily the null functional: in fact, for each $A\subseteq \omega$ we have $f(\bm{1}_A)=0$ if $A \in \mathcal{I}$ and $f(\bm{1}_A)=f(e-\bm{1}_{\omega\setminus A})=-f(\bm{1}_{\omega\setminus A})=0$ if $A\notin \mathcal{I}$. It follows that $0=f=f_0-f_0$, cf. Remark \ref{rmk:equivalentdefi}, hence the decomposition is unique. 

    Second, suppose that $\|f\|=2$ and that $f=g-h$ for some $g,h \in \mathrm{SL}(\mathcal{I})$. Thanks to Lemma \ref{lem:abramoich}, we have $2=\|f\|=2\sup\{g(\bm{1}_A)-h(\bm{1}_A): A\subseteq \omega\}$ and, at the same time, $1=\|g\|=\sup\{g(\bm{1}_A): A\subseteq \omega\}$. It follows that, for each $n \in \omega$, there exists $A_n\subseteq \omega$ such that 
    $$
    1-2^{-n}\le g(\bm{1}_{A_n})\le 1\,\,\text{ and }0\le h(\bm{1}_{A_n})\le 2^{-n}.
    $$
    Since $g,h \in \mathrm{SL}(\mathcal{I})$, it follows also that 
    $$
    0\le g(\bm{1}_{\omega\setminus A_n})\le 2^{-n}\,\,\text{ and }1-2^{-n}\le h(\omega\setminus \bm{1}_{A_n})\le 1.
    $$
    For each $A\subseteq \omega$, we obtain that 
 \begin{displaymath}
        \begin{split}
0\le (g\wedge h)(\bm{1}_A)&=\inf\{g(u)+h(v):\bm{1}_A=u+v,u,v\ge 0\}\\
&\le \inf\{g(\bm{1}_B)+h(\bm{1}_{A\setminus B}):B\subseteq A\}\\
&\le \inf\{g(\bm{1}_{(\omega\setminus A_n) \cap A})+h(\bm{1}_{A\setminus (\omega\setminus A_n)}):n \in \omega\}\\
&\le \inf\{g(\bm{1}_{\omega\setminus A_n})+h(\bm{1}_{A_n}):n \in \omega\}=0. 
\end{split}
    \end{displaymath}    
It follows, again by Lemma \ref{lem:abramoich}, that 
  \begin{displaymath}
        \begin{split}
\|g\wedge h\|=2\sup\{(g\wedge h)(\bm{1}_A): A\subseteq \omega\}=0. 
\end{split}
    \end{displaymath}   
Therefore $f=g-h$ for some $g,h\in \ell_\infty^{\,\prime}$ such that $g\wedge h=0$. We conclude by \cite[Theorem 1.5(b)]{MR2262133} that $g=f^+$ and $h=f^-$, hence the decomposition is unique (note that the definition of $g$ coincides with the one provided in \eqref{eq:definitiong}, hence also the definition of $h:=g-f$).
    
    

    Lastly, suppose $\|f\|<2$ and that $\mathcal{I}$ is not maximal, so that $1-\|f\|/2>0$ and there exist two distinct maximal ideals $\mathcal{J}_1$, $\mathcal{J}_2$ which extend $\mathcal{I}$. Then it is easy to see that the function $g$ defined in \eqref{eq:definitiong} with $f_0(x):=\mathcal{J}_1\text{-}\lim x$ is different from the one choosing $f_0(x):=\mathcal{J}_2\text{-}\lim x$ (note that they are both $\mathrm{S}_\mathcal{I}$-limits, as it follows by Theorem \ref{thm:secondrepresenation}). Hence there exist at least two decompositions of $f$. 
\end{proof}

The following corollary is immediate: 
\begin{cor}
    Let $\mathcal{I}$ be an ideal on $\omega$ and fix $f \in \ell_\infty^{\,\prime}$ such that $f(e)=f(\bm{1}_A)=0$ for all $A\in \mathcal{I}$. Then for every real $k \ge \|f\|/2$ there exist $g,h \in \mathrm{SL}(\mathcal{I})$ such that $f=k(g-h)$. 
\end{cor}
\begin{proof}
    If $f=0$ the claim is clear. Otherwise, fix $k\ge \|f\|/2$ (which is positive) and define $f_0:=f/k$. The claim follows applying Theorem \ref{thm:differenceSIlimits} to $f_0$. 
\end{proof}

\medskip

We conclude with the proof of Theorem \ref{thm:freedman}.
\begin{proof}[Proof of Theorem \ref{thm:freedman}] 
    The equality $\overline{c(\mathcal{I}^\star) \cap \ell_\infty}=c(\mathcal{I}) \cap \ell_\infty$ is well known, see e.g. \cite[Theorem 2.4]{MR2181783}. Hence, it is enough to show that $\mathscr{V}(\mathcal{I})=c(\mathcal{I}) \cap \ell_\infty$. 
    For, suppose that $x \in \mathscr{V}(\mathcal{I})$. Since $x \in \ell_\infty$, then $\Gamma_x(\mathcal{I})\neq \emptyset$, see e.g. \cite[Lemma 3.1(vi)]{MR3920799}. Suppose that $\eta_1,\eta_2 \in \Gamma_x(\mathcal{I})$. Thanks to Theorem \ref{thm:Iclusterpoints}, there exists maximal ideals $\mathcal{J}_1, \mathcal{J}_2$ containing $\mathcal{I}$ such that $\mathcal{J}_1\text{-}\lim x=\eta_1$ and $\mathcal{J}_2\text{-}\lim x=\eta_2$. It follows by 
    Theorem \ref{thm:secondrepresenation} 
    that $\eta_1=\eta_2$, hence $|\Gamma_x(\mathcal{I})|=1$. This implies that $x$ is $\mathcal{I}$-convergent, see e.g. \cite[Corollary 3.4]{MR3920799}, therefore $\mathscr{V}(\mathcal{I})\subseteq c(\mathcal{I}) \cap \ell_\infty$. 

    Viceversa, pick a sequence $x \in c(\mathcal{I}) \cap \ell_\infty$ and define $\eta:=\mathcal{I}\text{-}\lim x$. Of course, ${\mathcal{F}}\text{-}\lim x=\eta$ for every ultrafilter $\mathcal{F} \in \mathrm{Ult}(\mathcal{I})$. Pick also a $S_{\mathcal{I}}$-limit $f$. Thanks to Theorem \ref{thm:firstrepresenation}, there exists a normalized capacity $\rho: \mathscr{B}(\mathrm{Ult}(\mathcal{I}))\to \mathbf{R}$ which satisfies \eqref{eq:firstrepresetnation}. This implies that 
    $$
    f(x)=\int_{\mathrm{Ult}(\mathcal{I})} {\mathcal{F}}\text{-}\lim x \, \mathrm{d}\rho(\mathcal{F})
    =\int_{\mathrm{Ult}(\mathcal{I})} \eta \, \mathrm{d}\rho(\mathcal{F})=\eta, 
    $$
    so that $x \in \mathscr{V}(\mathcal{I})$. Hence $c(\mathcal{I}) \cap \ell_\infty\subseteq \mathscr{V}(\mathcal{I})$. 
    
    \smallskip

    For the second part, suppose that $\mathcal{I}$ is a $P$-ideal. Then $c(\mathcal{I})=c+c_{00}(\mathcal{I})$ (the case $\mathcal{I}=\mathcal{Z}$ can be found in \cite[Theorem 2.3]{MR954458}, the general case goes analogously). We conclude that $\mathscr{V}(\mathcal{I})=c(\mathcal{I})\cap \ell_\infty=(c+c_{00}(\mathcal{I}))\cap \ell_\infty=c+(c_{00}(\mathcal{I})\cap \ell_\infty)$. 
\end{proof}

\section{Further results}\label{sec:furtherresults}

As an application of Theorem \ref{thm:differenceSIlimits}, we can compute the distance of sequences in $\ell_\infty$ from the space of bounded $\mathcal{I}$-convergent sequences. To this aim, for each nonempty $Y\subseteq \ell_\infty$ and $x \in \ell_\infty$, define 
$$
\mathrm{dist}(x,Y):=\inf\{\|x-y\|: y \in Y\}.
$$
(The case $\mathcal{I}=\mathrm{Fin}$ of the next result can be found, e.g., in \cite[Proposition 1.18]{MR1727673}.)
\begin{prop}\label{prop:distance}
Let $\mathcal{I}$ be an ideal on $\omega$. Then 
\begin{equation}\label{eq:claimed}
\forall x \in \ell_\infty, \quad \mathrm{dist}(x,c(\mathcal{I}) \cap \ell_\infty)=\frac{1}{2}\left(\mathcal{I}\text{-}\limsup x-\mathcal{I}\text{-}\liminf x\right).
\end{equation}
\end{prop}
\begin{proof}
    Fix $x \in \ell_\infty$ and define $\eta_0:=\frac{1}{2}(\eta_++\eta_-)$ and $\delta_0:=\frac{1}{2}(\eta_+-\eta_-)$, where $\eta_+:=\mathcal{I}\text{-}\limsup x$ and $\eta_-:=\mathcal{I}\text{-}\liminf x$. Set also $Y:=c(\mathcal{I})\cap \ell_\infty$ and recall that $Y$ is closed, see e.g. Theorem \ref{thm:freedman}. If $x$ is a bounded $\mathcal{I}$-convergent sequence then both sides of \eqref{eq:claimed} are zero. Hence, we suppose hereafter that $x\notin Y$. 
    
    On the one hand, it is known (and easy to check) that, if $y \in Y$ is $\mathcal{I}$-convergent to $\kappa$, then $\Gamma_{x-y}(\mathcal{I})=\Gamma_{x-\kappa e}(\mathcal{I})=\Gamma_x(\mathcal{I})-\kappa$, cf. also \cite[Proposition 3.2]{MR4428911} for a more general result in this direction. Hence 
    \begin{displaymath}
        \begin{split}
            \mathrm{dist}(x,Y) &\ge \inf\left\{ \max \{|\eta|:\eta \in \Gamma_{x-y}(\mathcal{I})\}: y \in Y \right\} \\
    &=\inf\{\max \{|\eta|:\eta \in \Gamma_{x}(\mathcal{I})-\kappa\}: \kappa \in \mathbf{R}\}\\
    &\ge \inf\{\max\{|\eta_+-\kappa|,|\eta_--\kappa|\}:\kappa\in \mathbf{R}\}\ge \delta_0.
        \end{split}
    \end{displaymath}

    On the other hand, it follows by a consequence of Hahn--Banach theorem that there exists $f \in \ell_{\infty}^{\,\prime}$ such that 
    $$
    \|f\|=1,\,\, f(x)=\mathrm{dist}(x,Y),\,\,\text{and}\,\, f[Y]=\{0\},
    $$
    see e.g. \cite[Theorem 2.3.22]{MR3823238}. 
    In particular, since $e \in Y$, we have $f(e)=0$. It follows by Theorem \ref{thm:differenceSIlimits} that there exists (unique) $g,h \in \mathrm{SL}(\mathcal{I})$ such that $2f=g-h$. Considering that $\eta_- \le w(x)\le \eta_+$ for every $w \in \mathrm{SL}(\mathcal{I})$, see Corollary \ref{cor:liminflimsup}, we obtain that 
    $\mathrm{dist}(x,Y)=f(x)=(g(x)-h(x))/2 \le (\eta_+-\eta_-)/2=\delta_0$. 
    %
\end{proof}

\begin{rmk}
    It is worth noting that the upper bound $\mathrm{dist}(x,Y)\le \delta_0$ can be obtained also by elementary means without the aid of $\mathrm{S}_{\mathcal{I}}$-limits. To this aim, for each $k \in \omega$, define $A_k:=\{n \in \omega: |x_n-\eta_0|\ge \delta_0(1+2^{-k})\}$, and note that $A_k \in \mathcal{I}$. For each $k \in \omega$, define also the sequence $y^k\in Y$ by $y^k_n:=x_n$ if $n \in A_k$ and $y^k_n:=\eta_0$ otherwise. It follows that 
    \begin{displaymath}
        \begin{split}
            \mathrm{dist}(x,Y) &\le \inf\{\|x-y^k\|: k \in \omega\}\\
            &=\inf\{\sup\{|x_n-\eta_0|: n \in \omega\setminus A_k\}: k \in \omega\}\\
            &\le \inf\{\delta_0(1+2^{-k}): k\in \omega\}=\delta_0.
    \end{split}
    \end{displaymath}
\end{rmk}

For our last application, 
a point $\eta \in \mathbf{R}$ is a $\mathcal{I}$\emph{-limit point} of a sequence $x \in \ell_\infty$ if there exists a strictly increasing sequence $(a_n: n \in \omega)$ in $\omega$ such that $\{a_n: n \in \omega\}\notin \mathcal{I}$ and $\lim_n x_{a_n}=\eta$; see e.g. \cite{MR1181163, MR1416085} in the case $\mathcal{I}=\mathcal{Z}$. We denote by $\Lambda_x(\mathcal{I})$ the set of $\mathcal{I}$-limit points of $x \in \ell_\infty$. Analogously to Theorem \ref{thm:Iclusterpoints}, we provide a characterization of $\mathcal{I}$-limit points, which generalizes \cite[Theorem 2(2)]{MR1372186}.

\begin{thm}\label{thm:Ilimitpoints}
    Let $\mathcal{I}$ be an ideal on $\omega$, and fix $x \in \ell_\infty$ and $\eta \in \mathbf{R}$. Then $\eta \in \Lambda_x(\mathcal{I})$ if and only if there exists $A\subseteq \omega$ such that $\hat{x}[\,\tilde{A}\,]=\{\eta\}$ and $\tilde{A} \cap \mathrm{Ult}(\mathcal{I})\neq \emptyset$. 
\end{thm}
\begin{proof}
    It follows by Lemma \ref{lem:bndensityzero} and Lemma \ref{lem:bnsubseqconvergence}.
\end{proof}

Differently from the case of $\mathcal{I}$-cluster points, the behavior of $\mathcal{I}$-limit points is really wild: for instance, there exists $x \in \ell_\infty$ such that $\Lambda_x(\mathcal{Z})=\emptyset$, see \cite[Example 4]{MR1181163}; cf. also \cite{FKL24} and references therein. 
However, if the complexity of $\mathcal{I}$ is sufficiently low, then we can show that the notions of $\mathcal{I}$-cluster points and $\mathcal{I}$-limit points coincide, hence recovering \cite[Theorem 2.3]{MR3883171} in the real case: 
\begin{prop}
Let $\mathcal{I}$ be a $F_\sigma$-ideal on $\omega$. Then $\Lambda_x(\mathcal{I})=\Gamma_x(\mathcal{I})$ for all $x \in \ell_\infty$. 
\end{prop}
\begin{proof}
    Fix $x \in \ell_\infty$. It is known (and it follows also by Theorem \ref{thm:Iclusterpoints} and Theorem \ref{thm:Ilimitpoints}) that every $\mathcal{I}$-limit point of $x$ is also an $\mathcal{I}$-cluster point of $x$. Hence it is sufficient to show that $\Gamma_x(\mathcal{I}) \subseteq \Lambda_x(\mathcal{I})$. 
For, fix $\eta \in \Gamma_x(\mathcal{I})$. Thanks to Theorem \ref{thm:Iclusterpoints}, there exists a free ultrafilter $\mathcal{F}_0 \in \mathrm{Ult}(\mathcal{I})$ such that $\hat{x}(\mathcal{F}_0)=\eta$. 
It follows that 
$$
A_k:=\{n\in \omega: |x_n-\eta|<2^{-k}\} \in \mathcal{F}_0
$$
for each $k \in \omega$. In particular, $A_k \notin \mathcal{I}$. 

At this point, since $\mathcal{I}$ is a $F_\sigma$-ideal, there exists a sequence $(F_n: n \in \omega)$ of closed subsets of $\mathcal{P}(\omega)$ such that $\mathcal{I}=\bigcup_n F_n$. Replacing $F_n$ with $\bigcup_{k\le n}\bigcup_{A \in F_k}\mathcal{P}(A)$, we can assume, in addition, that $F_n$ is hereditary and $F_n \subseteq F_{n+1}$ for all $n\in \omega$. 
Note also that, for each $k,n\in \omega$ and $S \in \mathrm{Fin}$, we have $A_k\setminus S \notin \mathcal{I}$, so that $A_k \setminus S \notin F_n$. 


Lastly, define recursively the increasing sequence $(m_k: k\in \omega)$ in $\omega$ such that $m_0:=0$ and, for each $k\ge 1$, $m_k$ is the smallest integer $m>m_{k-1}$ such that $A_k \cap [m_{k-1},m) \notin F_k$ (this is well defined because $B_k:=A_k \setminus [0,m_k) \notin F_k$ and $F_k$ is hereditary closed). It follows by construction that the subsequence $(x_n: n \in B)$ converges to $\eta$, where 
$
B:=\bigcup_{k}B_k.
$ 
In addition, $B\notin \mathcal{I}$: indeed, in the opposite, there would exist $k \in \omega$ such that $B \in F_k$ hence, in particular, $B_k \in F_k$ which is impossible. This proves that $\eta$ is a $\mathcal{I}$-limit point of $x$, completing the proof. 
\end{proof}


\section{Concluding Remarks}

We leave as an open question for the interested reader to characterize the class of ideals $\mathcal{I}$ on $\omega$ which satisfy the claim of Theorem \ref{thm:strongdiameter}. 

\bibliographystyle{amsplain}

\end{document}